\documentclass[a4paper,11pt,twoside]{article}
\usepackage{pstricks,pstricks-add,pst-math,pst-xkey}

\date{}
\usepackage{amsmath}
\usepackage{amssymb}
\usepackage{amsthm}

\newcommand{\C}{\mathbb{C}}

\newcommand{\g}{\mathfrak{g}}

\newcommand{\bdm}{\begin{displaymath}}
\newcommand{\edm}{\end{displaymath}}

\theoremstyle{definition}

\newtheorem{thm}{Theorem}[section]

\newtheorem{cor}{Corollary}[section]

\title{Index and nullity of a family of harmonic tori in the sphere}
\author{R. Pacheco}

\begin{document}

\maketitle

\begin{abstract}
In this note we investigate the $(E)$-index and the $(E)$-nullity of vacuum solutions from the two-torus into the two sphere.
\end{abstract}



\emph{Departamento de Matem\'{a}tica, Universidade da Beira
Interior, Rua Marqu\^{e}s d'\'{A}vila e Bolama 6201-001
Covilh\~{a} - Portugal}

\emph{email: rpacheco@ubi.pt }
\section{Introduction}

Harmonic maps are the critical points of the energy
$$E(\varphi)=\frac12\int_M \mathrm{trace}_g\varphi^*h\,d\mathrm{vol}_M.$$
on the space of smooth maps between two Riemannian manifolds, $(M,g)$ and $(N,h)$. The $(E)$-index of a harmonic map $\varphi$ is defined as  the dimension of the largest subspace of $\Gamma(\varphi^{-1}TN)$ on which the Hessian $H(E)_\varphi$ is negative definite. Hence it gives  a measure of the instability of $\varphi$. On the other hand, the $(E)$-nullity of $\varphi$ plays an important rule in the determination of the structure of the subspace of harmonic maps between two Riemannian manifolds in a neighborhood of $\varphi$ \cite{Ur}. However, the calculation of these quantities is nontrivial even for the simplest cases, such as the identity map \cite{Ur}.

In this note we investigate the Morse index and the nullity of a certain family of harmonic maps from a two-torus $T^2$ to the round two-sphere $S^2$:  the vacuum solutions. More precisely, we show that the problem of determining such quantities amounts to the problem of counting the number of lattice points on and inside an ellipse whose dimensions are controlled by the energy of the vacuum solution (Theorem \ref{main}).

 In general, vacuum solutions from  $T^2$ to a symmetric space $N$ are the simplest harmonic maps of semisimple finite type between such spaces \cite{BP}. Burstall and Pedit  \cite{BP} proved  that any harmonic map from $T^2$ to $N$ of semisimple finite type lies in the dressing orbit of some vacuum solution. On the other hand, from \cite{BFP} we know that all non-conformal harmonic maps $\varphi:T^2\to S^2$ are of semisimple finite type (in this particular setting, conformal harmonic maps are all holomorphic, which are known to be weakly stable and whose nullity is given by the dimension of the space of holomorphic sections of the pull-back bundle $\varphi^{-1}TN$ \cite{Ur}). Hence,  once we know how dressing actions behave with respect to the $(E)$-index and $(E)$-nullity, it seems that the results we present here can be relevant in the approach to the general case of non-confomal harmonic maps from $T^2$ to $S^2$.

\section{The second variation formula}

Let $(M,g)$ and $(N,h)$ be Riemannian manifolds, $M$ compact, and $\varphi:M\to N$ a harmonic map.
The Hessian of the energy $E$ at $\varphi$ is given by \cite{Ur}
$$H(E)_\varphi(V,W)=\int_Mh(J_\varphi(V),W)\,d\mathrm{vol}_M,$$
where $V,W\in \Gamma(\varphi^{-1}TN)$, the space of variation vector fields along $\varphi$, and $J_\varphi$ is the Jacobi operator
$$J_\varphi (V)=\Delta_\varphi V -\mathrm{trace}R^N(d\varphi,V)d\varphi.$$ Here $\Delta_{\varphi}V=-\mathrm{trace}\nabla dV$ is the Laplacian on $\Gamma(\varphi^{-1}TN)$ and the sign convention on the curvature $R$ is
$$R(X,Y)Z=-\nabla_X\nabla_YZ+\nabla_Y\nabla_XZ+\nabla_{[X,Y]}Z.$$

 Since $J_\varphi$ is a self-adjoint elliptic differential operator acting on $\Gamma(\varphi^{-1}TN)$, the eigenvalues of $J_\varphi$ have finite multiplicities and form a bounded below discrete set without accumulation points: $\lambda_1(\varphi)<\lambda_2(\varphi)<\lambda_3(\varphi)<\ldots$ The space $\Gamma(\varphi^{-1}TN)$ splits as the orthogonal sum of the eigenspaces $V_{\lambda_i(\varphi)}$ of $J_\varphi$. The $(E)$-\emph{nullity} of $\varphi$ is the dimension of the kernel $V_0(\varphi)$ of $J_\varphi$. The $(E)$-\emph{index} of $\varphi$ is the dimension of the largest subspace of $\Gamma(\varphi^{-1}TN)$ on which the Hessian $H(E)_\varphi$ is negative definite, that is,
 $\mathrm{index}(\varphi)=\sum_{\lambda_i<0}\mathrm{dim} V_{\lambda_i(\varphi)}.$

In the case $M$ is a Riemann surface, the energy is conformally invariant for the domain metric. Hence, the harmonicity, $(E)$-nullity and $(E)$-index of a map are also conformally invariant for the domain metric. In local conformal complex coordinate $z=x+iy$, the Jacobi operator is given by
$$J_\varphi(V)=-4\mu^2\Big\{\frac{\nabla}{\partial \bar{z}}\frac{\nabla}{\partial {z}}V+R^N\big(\frac{\partial \varphi}{\partial \bar{z}},V\big)\frac{\partial \varphi}{\partial \bar{z}}\Big\},$$
where the smooth function $\mu$ is such that $e_1=\mu \frac{\partial}{\partial x},e_2=\mu \frac{\partial}{\partial y}$  is a local orthonormal frame.

\section{Symmetric spaces}

Next we recall some well known facts about symmetric spaces. For details see \cite{BR}.

Let $G$ be a compact, connected and semisimple matrix Lie group, and $N=G/H$ a symmetric space associated to the involution $\sigma$ on $G$. Fix a base point $x_0\in N$ and a $G$-invariant metric on $N$. Consider the corresponding symmetric decomposition of the Lie
algebra $\g$ of $G$: $\mathfrak{g}=\mathfrak{h}\oplus \mathfrak{m}$, where $\mathfrak{h}$ is the Lie algebra of $H$ and
\begin{equation}\label{s}
 [\mathfrak{h},\mathfrak{m}]\subset \mathfrak{m},\quad [\mathfrak{m},\mathfrak{m}]\subset \mathfrak{h}.
\end{equation}
Denote
by $P_\mathfrak{m}$ (resp. $P_\mathfrak{h}$) the projection onto
$\mathfrak{m}$ (resp. $\mathfrak{h}$). The $\mathfrak{g}$-valued
one form $\beta$ on $N$ defined by
$$\beta_x(X)=\mathrm{Ad}_gP_\mathfrak{m}(\mathrm{Ad}_{g^{-1}})\xi\quad
\mbox{if}\quad X=\frac{d}{dt}\Big{|}_{t=0}\exp t\xi\cdot x$$ and
$x=g\cdot x_0$  is called the \emph{Maurer-Cartan form} of the
symmetric space $N$. The Maurer-Cartan form $\beta$ of $N$ gives a vector bundle
isomorphism $\beta:TN\to [\mathfrak{m}]$, where the fiber of
$[\mathfrak{m}]$ at $x=g\cdot x_0$, $[\mathfrak{m}]_x$, is given
by $\mathrm{Ad}_g(\mathfrak{m})$.

Since $N$ is symmetric, the
corresponding Levi-Civita connection on $[\mathfrak{m}]$ is given
by flat differentiation followed by the projection onto
$[\mathfrak{m}]$:
\begin{equation}\label{nabla}
\beta\circ\nabla =P_{[\mathfrak{m}]}\circ
d\circ\beta.\end{equation} On the other hand, the curvature $R^N$ of $\nabla$ satisfies
\begin{equation}\label{R}
\beta\circ
R^N=\frac12\big[(1-P_{[\mathfrak{m}]})[\beta\wedge\beta],\beta\big]=
\frac12\big[[\beta\wedge\beta],\beta\big].
\end{equation}
%We also have
%$$d\beta=(1-\frac12P_{[\mathfrak{m}]})[\beta\wedge\beta]=[\beta\wedge\beta].$$

In the study of harmonic maps into symmetric spaces, it is useful to consider the  Cartan embedding  $\iota: G/H\to G$, which is defined by $\iota(g\cdot x_0)=\sigma(g)g^{-1}$. It is well known that $\iota$ is totally geodesic, so that a smooth map $\varphi:M\to G/H$ is harmonic if and only if $\iota\circ\varphi:M\to G$ is harmonic.

\section{Vacuum solutions}

Let $N=G/H$ be a symmetric space. A smooth map $\varphi:M\to N$ is harmonic if, and only if, the pull-back of the Maurer-Cartan form is co-closed (\cite{Ra}): $d^*\varphi^{-1}\beta=0$. Now, suppose we have a framing of $\varphi$, that is, a smooth map $F:M\to G$ such that $\varphi=F\cdot x_0$. Set $\alpha=F^{-1}dF$ and decompose it with respect to $\g=\mathfrak{h}\oplus\mathfrak{m}$:  $\alpha=\alpha_\mathfrak{h}+\alpha_{\mathfrak{m}}$. Taking account that
\begin{equation}\label{-1}
\varphi^{-1}\beta=\mathrm{Ad}_F\alpha_\mathfrak{m},
 \end{equation}
the harmonicity condition on $\varphi$ is equivalent to
\begin{equation}\label{har}
d^*\alpha_\mathfrak{m}+[\alpha\wedge *\alpha_\mathfrak{m}]=0.
\end{equation}

One of the simplest cases  occurs  when we take $A\in\mathfrak{m}^\C$ such that $[A,\overline{A}]=0$. Define $F_A:\C\to G$ by
$F_{A}(z)=\exp(zA+\bar{z}\overline{A})$. It is clear that
\begin{equation}\label{ff}\alpha=F_A^{-1}dF_A=Adz+\overline{A}d\bar{z}
\end{equation}
 satisfies (\ref{har}). Hence we get a harmonic map $\varphi_A=F_A\cdot x_0:\C\to G/H$, which is called a \emph{vacuum solution.}

\subsection{Vacuum solutions from the torus to the sphere}

% for all $\lambda\in S^1$, the connection  $d+\alpha_\lambda$, with $\alpha_\lambda=\lambda^{-1}Adz+\lambda\overline{A}d\bar{z}$, is flat
%\begin{defn}\emph{ A \emph{vacuum solution} is an extending framing of the form
%$F^{\eta_A}$ where $\eta_A=\lambda^{-1}A$ with
%.}
%\end{defn}
%In this simple case, we can perform the Iwasawa decomposition
%explicitly:
%$$\exp(z\eta_A)=\exp(z\lambda^{-1}A+\bar{z}\lambda\overline{A})\exp(-\bar{z}\lambda\overline{A}),$$
%so that
%$$F^{\eta_A}(z)=\exp(z\lambda^{-1}A+\bar{z}\lambda\overline{A}).$$
%Of course that any vacuum solution corresponds to a harmonic map
%of finite type. Moreover, we have:
%$
%
%\begin{prop}
%\emph{$g\#F^{\eta_A}$ is of finite type if and only if, for some
%$d=1\mod 2$ there is $\xi\in\Lambda_d$ such that
%$\xi_{-d}=\mathrm{Ad}_{g(0)}A$ and $$[\xi,\mathrm{Ad}_gA]=0.$$}
%\end{prop}

The round sphere $S^2$ is a symmetric $SU(2)$-space with stabilizers conjugate to $S(U(1)\times U(1))$. For a convenient choice of base-point, the
 symmetric decomposition $\mathfrak{su}(2)=\mathfrak{h}\oplus \mathfrak{m}$ is given by
\begin{align*}
\mathfrak{h}=\Big\{  \left(%
\begin{array}{cc}
  ix & 0 \\
  0 & -ix \\
\end{array}%
\right),\,\,x\in \mathbb{R}\Big\},\quad\quad\mathfrak{m}=\Big\{  \left(%
\begin{array}{cc}
  0 & z \\
  -\bar{z} & 0 \\
\end{array}%
\right),\,\,z\in \mathbb{C}\Big\}
\end{align*}
and the corresponding
  involution $\sigma$ is given by conjugation by
$$Q= \left(%
\begin{array}{cc}
  1 & 0 \\
  0 & -1 \\
\end{array}%
\right).
$$
The round metric is that induced by the $SU(2)$-invariant inner product on  $\mathfrak{su}(2)$ given by $(X,Y)=-\frac12\mathrm{trace}XY$.

Consider $A\in\mathfrak{m}^\C$. Then $A$ is of the
form
$$A=\left(%
\begin{array}{cc}
  0 & \alpha \\
  \beta & 0 \\
\end{array}%
\right)
$$
where $\alpha,\beta\in \C$. Since $\overline{A}=-A^*$ we get:
$$[A,\overline{A}]=\left(%
\begin{array}{cc}
  \alpha\bar{\alpha}-\beta\bar{\beta} & 0 \\
  0 &  \beta\bar{\beta}- \alpha\bar{\alpha}\\
\end{array}%
\right).
$$
Hence, $[A,\overline{A}]=0$ if and only if
\begin{equation}\label{ab}
\alpha\bar{\alpha}=\beta\bar{\beta}.
\end{equation}

 Next we describe
which of these matrices give rise to harmonic maps from a two-torus to $S^2$. We shall use the  Cartan embedding of $S^2$ to
identify it with a totally geodesic submanifold of $SU(2)$. In
this case, the harmonic map corresponding to
$F_{A}$ is given by:
$$\varphi_A(z)=\sigma(F_A){F_A}^{-1}=\exp{\big(-2zA-2\bar{z}\overline{A}\big)}.$$
Notice that the condition $[A,\overline{A}]=0$ implies that $A$ is semisimple. Hence there exists  a  constant matrix $G$ such that
$$\exp{\big(-2zA-2\bar{z}\overline{A}\big)}=G\left(%
\begin{array}{cc}
  e^{2\bar{z}\overline{\sqrt{\alpha\beta}}-2z\sqrt{\alpha\beta}} & 0 \\
  0 & e^{2z\sqrt{\alpha\beta}-2\bar{z}\overline{\sqrt{\alpha\beta}}} \\
\end{array}%
\right)G^{-1}.$$ It is now easy to check that $\varphi$ is double
periodic with periods $\omega_1$ and $\omega_2$ if
and only if
\begin{equation}\label{popo}
\sqrt{\alpha\beta}=\pi i \,\,\frac{\overline{\omega}_2 n+\overline{\omega}_1 m}{\overline{\omega}_2\omega_1-\omega_2\overline{\omega}_1}\end{equation} with
$n,m\in\mathbb{Z}$.

Let $T^2$ be the corresponding torus. The energy of such harmonic map is given by
\begin{align}\label{en}
 E(\varphi_A)&=\frac12\int_{T^2}\Big|\varphi_A^{-1}\frac{\partial \varphi_A}{\partial x}\Big|^2+\Big|\varphi_A^{-1}\frac{\partial \varphi_A}{\partial y}\Big|^2dxdy=4 \pi^2 \frac{|\overline{\omega}_2 n+\overline{\omega}_1 m|^2}{|\overline{\omega}_2\omega_1-\omega_2\overline{\omega}_1|}.
 \end{align}

\section{Index and Nullity of Vacuum Solutions}

Let $N=G/H$ be a symmetric space with Maurer-Cartan form $\beta$ and $\varphi_A:\C\to N$ a vacuum solution with $\varphi_A=F_A\cdot x_0$.  If $V$ is a section of $\varphi_A^{-1}TN$, then there
 exists a smooth map $v:\C\to \mathfrak{m}$ such that $\beta(V)=F_AvF_A^{-1}$.
Moreover, from (\ref{-1}) we have:
$$\beta \big(\frac{\partial
\varphi_A}{\partial z}\big)=F_AAF_A^{-1}\quad \mbox{and} \quad
\beta\big(\frac{\partial
\varphi_A}{\partial \bar{z}}\big)=F_A\overline{A}F_A^{-1}.$$
Hence, by using  (\ref{s}), (\ref{nabla}), (\ref{R}) and (\ref{ff}):\begin{align*}
\beta\Big(\frac{\nabla}{\partial \bar{z}}\frac{\nabla}{\partial
{z}}V+&R\Big(\frac{\partial \varphi_A}{\partial
\bar{z}},V\Big)\frac{\partial \varphi_A}{\partial
z}\Big)=P_{[\mathfrak{m}]}\frac{\partial}{\partial
\bar{z}}\beta\Big(\frac{\nabla}{\partial
{z}}V\Big)+\big[\big[\beta\big(\frac{\partial \varphi_A}{\partial
\bar{z}}\big),\beta(V)\big],\beta\big(\frac{\partial
\varphi_A}{\partial z}\big)\big]\\
&=P_{[\mathfrak{m}]}\frac{\partial}{\partial\bar{z}}
P_{[\mathfrak{m}]}\frac{\partial}{\partial z}F_A vF_A^{-1} +
\big[\big[F_A \overline{A}F_A^{-1},F_A vF_A^{-1}\big],
F_A AF_A^{-1}\big]\\
%&=P_{[\mathfrak{m}]}\frac{\partial}{\partial\bar{z}}\Big\{
%P_{[\mathfrak{m}]}F\big[F^{-1}\frac{\partial F}{\partial {z}},
%v\big]F^{-1}-F vF^{-1}\Big\} +
%F\big[\big[\overline{A}, v\big], A\big]F^{-1}\\
&=P_{[\mathfrak{m}]}\frac{\partial}{\partial\bar{z}}\Big\{
P_{[\mathfrak{m}]}\underbrace{F_A\big[{A}, v\big]F_A^{-1}}_{\in
[\mathfrak{h}]}+F_A \frac{\partial v}{\partial {z}}F_A^{-1}\Big\} +
F_A\big[\big[\overline{A}, v\big], A\big]F_A^{-1}\\
&=P_{[\mathfrak{m}]}\frac{\partial}{\partial\bar{z}}F_A
\frac{\partial v}{\partial
{z}}F_A^{-1}+F_A\big[\big[\overline{A}, v\big], A\big]F_A^{-1}\\
&=F_A\frac{\partial^2v}{\partial\bar{z}\partial z}
F_A^{-1}-F_A\big[A,\big[\overline{A}, v\big]\big]F_A^{-1}.
\end{align*}
Then
$$\beta(J_{\varphi_A}(V))=-4\mu^{2}F_A\Big\{\frac{\partial^2v}{\partial\bar{z}\partial z}
-\big[A,\big[\overline{A}, v\big]\big]\Big\}F_A^{-1}.$$

Now, let us come back to the particular case of harmonic tori in the sphere. Suppose that $\varphi_A$ is periodic with periods $\omega_1$ and $\omega_2$. Consider on the corresponding torus $T^2$ the flat metric $\mu=1$. Since $F(z+\omega_i)=F(z)F(\omega_i)$ and $\varphi_A(z)=F^{-2}(z)$, the periodicity of $\varphi$ implies that $F(\omega_i)$ belongs to the center $Z=\{\pm Id\}$ of $SU(2)$. Hence, a vector field $V\in \Gamma{(\varphi_A^{-1}TN)}$ is periodic if, and only if, $v$ is periodic. So, in this case, the index and the nullity of $J_{\varphi_A}$ are given by the index and nullity of the elliptic operator
$$\mathcal{J}_{\varphi_A}(v)= -\frac{\partial^2v}{\partial\bar{z}\partial z}+\big[A,\big[\overline{A}, v\big]\big]$$ acting on the space of functions from $T^2$ to $\mathfrak{m}$.

Set
$$A=\left(%
\begin{array}{cc}
  0 & \alpha \\
  \beta & 0 \\
\end{array}%
\right)\,\,\mbox{and}\quad v=\left(%
\begin{array}{cc}
  0 & f \\
  -\bar{f} & 0 \\
\end{array}%
\right),
$$
with $\alpha\bar{\alpha}=\beta\bar{\beta}$ and
$f:\mathbb{C}\to\mathbb{C}$ periodic with periods $\omega_1=\omega_{1x}+i\omega_{1y}$ and $\omega_2=\omega_{2x}+i\omega_{2y}$. The smooth function $v$ is an eigenvector of $\mathcal{J}_{\varphi_A}$ with eigenvalue $\lambda$ if and only if
\begin{equation}\label{valpp}
\frac{\partial^2 f}{\partial^2x}+\frac{\partial^2
f}{\partial^2y}+8f\alpha\bar{\alpha}+8\bar{f}\alpha\bar{\beta}=-4\lambda f.
\end{equation}
%We have:
%$$\big[
%A,\big[\overline{A}, v\big]\big]=\left(%
%\begin{array}{cc}
%  0 & -2f\alpha\bar{\alpha}-2\bar{f}\alpha\bar{\beta} \\
%2\bar{f}\alpha\bar{\alpha}+2{f}\bar{\alpha}{\beta}& 0 \\
%\end{array}%
%\right).$$
We express a Fourier expansion of $f$ as
\begin{equation*}
f(z=x+iy)=\sum_{k,l\in\mathbb{Z}}f_{k,l}e^{2k\pi i \frac{\omega_{1x}x+\omega_{1y}y}{|\omega_1|^2}}e^{2l\pi i \frac{\omega_{2x}x+\omega_{2y}y}{|\omega_2|^2}}.
\end{equation*}
%Write
%\begin{equation}\label{vs}
%v_{k,l}=\left(%
%\begin{array}{cc}
%  0 & f_{k,l} \\
%  -\overline{f_{k,l}} & 0 \\
%\end{array}%
%\right).
%\end{equation}
Equation (\ref{valpp}) becomes  equivalent to
\begin{equation}\label{1}
\Big\{\Big(k\frac{\omega_{1x}}{|\omega_{1}|^2}+l\frac{\omega_{2x}}{|\omega_{2}|^2}\Big)^2+\Big(k\frac{\omega_{1y}}{|\omega_{1}|^2}+l\frac{\omega_{2y}}{|\omega_{2}|^2}\Big)^2-\frac{2}{\pi^2}\alpha\bar{\alpha}-\frac{1}{\pi^2}\lambda\Big\}f_{k,l}=\frac{2}{\pi^2}\overline{f_{-k,-l}}
\alpha\overline{\beta}
\end{equation}
for all integers $k,l$. By taking the conjugate of this we obtain
\begin{equation}\label{2}
\Big\{\Big(k\frac{\omega_{1x}}{|\omega_{1}|^2}+l\frac{\omega_{2x}}{|\omega_{2}|^2}\Big)^2+\Big(k\frac{\omega_{1y}}{|\omega_{1}|^2}+l\frac{\omega_{2y}}{|\omega_{2}|^2}\Big)^2-\frac{2}{\pi^2}\alpha\bar{\alpha}-\frac{1}{\pi^2}\lambda\Big\}\overline{f_{-k,-l}}=\frac{2}{\pi^2}{f_{k,l}}
\overline{\alpha}{\beta}
\end{equation}
From (\ref{1}) and (\ref{2}) results that, since $\alpha\bar{\alpha}=\beta\bar{\beta}$,
\begin{equation}\label{3}
\Big(k\frac{\omega_{1x}}{|\omega_{1}|^2}+l\frac{\omega_{2x}}{|\omega_{2}|^2}\Big)^2+\Big(k\frac{\omega_{1y}}{|\omega_{1}|^2}+l\frac{\omega_{2y}}{|\omega_{2}|^2}\Big)^2-\frac{2}{\pi^2}\alpha\bar{\alpha}-\frac{1}{\pi^2}\lambda=\pm\frac{2}{\pi^2}\alpha\bar{\alpha}
\end{equation}
Set $$\theta(k,l)=\Big(k\frac{\omega_{1x}}{|\omega_{1}|^2}+l\frac{\omega_{2x}}{|\omega_{2}|^2}\Big)^2+\Big(k\frac{\omega_{1y}}{|\omega_{1}|^2}+l\frac{\omega_{2y}}{|\omega_{2}|^2}\Big)^2$$ and observe, from (\ref{ab}), (\ref{popo}), and (\ref{en}),  that
\begin{equation}\label{sl}
\alpha\bar{\alpha}= \frac{E(\varphi_A)}{4|\overline{\omega}_2\omega_1-\omega_2\overline{\omega}_1|}.
\end{equation}
\begin{thm}\label{main}
\emph{The $(E)$-nullity and the $(E)$-index of the vacuum solution $\varphi_A:T^2\to S^2$ are given by}
\begin{align}\label{nul}
\mathrm{nullity}(\varphi_A)&=1+\#\Big\{(k,l)\in\mathbb{Z}^2:\, \theta(k,l)=\frac{E(\varphi_A)}{\pi^2 |\overline{\omega}_2\omega_1-\omega_2\overline{\omega}_1|}\Big\};\\\label{index}
\mathrm{index}(\varphi_A)&=\#\Big\{(k,l)\in\mathbb{Z}^2:\, \theta(k,l)<\frac{E(\varphi_A)}{\pi^2 |\overline{\omega}_2\omega_1-\omega_2\overline{\omega}_1|}\Big\}.
\end{align}
\end{thm}
\begin{proof}
From (\ref{1}), (\ref{3}) and (\ref{sl}), it follows that $V$ is a Jacobi field along $\varphi$ if and only if $v$ is of the form
$$v=v_0+\!\!\!\!\!\!\!\!\sum_{\theta(k,l)=\frac{E(\varphi_A)}{\pi^2 |\overline{\omega}_2\omega_1-\omega_2\overline{\omega}_1|}}\!\!\!\!\!\!\!\!v_{k,l}e^{2k\pi i \frac{\omega_{1x}x+\omega_{1y}y}{|\omega_1|^2}}e^{2l\pi i \frac{\omega_{2x}x+\omega_{2y}y}{|\omega_2|^2}}.$$
with $v_0\in \mathfrak{m}\cap\ker \mathrm{ad}_A\circ\mathrm{ad}_{\overline{A}} $ and
\begin{equation*}\label{vs} v_{k,l}=\left(%
\begin{array}{cc}
  0 & f_{k,l} \\
  -\overline{f_{-k,-l}} & 0 \\
\end{array}%
\right)
\end{equation*}
satisfying (\ref{1}) for $\lambda=0$. Hence,
taking account that $\ker \mathrm{ad}_A\circ\mathrm{ad}_{\overline{A}}$ is one-dimensional, we obtain formula (\ref{nul}) for the nullity of $\varphi_A$.

Formula (\ref{index})  for the index of $\varphi_A$ can be obtained similarly.

\end{proof}

So the problem of finding the nullity and index of vacuum solutions is reduced to the  problem of counting the lattice points on and inside an ellipse. For example:

Set $$A(x)=\#\big\{(k,l)\in\mathbb{Z}^2: ak^2+blk+cl^2< x\big\},$$
and $D\equiv 4ac-b^2>0$.
It is well known  \cite{Co} that $A(x)=\frac{2\pi}{\sqrt{D}}x+O(x^c)$ for some $c<1$; we have:
\begin{cor}
\emph{For vacuum solutions,
$$\lim_{E(\varphi_A)\to\infty}\frac{\mathrm{index}(\varphi_A)}{{E(\varphi_A)}}=\frac{1}{2\pi\sin^2{(\angle\omega_1\omega_2)}},$$
where $\angle\omega_1\omega_2$ denotes the angle between $\omega_1$ and $\omega_2$. }  \end{cor}
\begin{proof}
  Taking account that $$|\overline{\omega}_2\omega_1-\omega_2\overline{\omega}_1|=2|\omega_1||\omega_2|\sin (\angle\omega_1\omega_2),\, \mbox{and}\,\,\,\, \omega_{1x}\omega_{2x}+\omega_{1y}\omega_{2y}=|\omega_1||\omega_2|\cos (\angle\omega_1\omega_2),$$
  the results follows by straightforward computation.
\end{proof}
%$$-(k^2+l^2)f_{k,l}+8f_{k,l}\alpha\bar{\alpha}+8\overline{f_{-k,-l}}\alpha\bar{\beta}=0\Leftrightarrow
%(k^2+l^2-8\alpha\bar{\alpha})
%f_{k,l}=-8\overline{f_{-k,-l}}\alpha\bar{\beta},$$ for all
%$k,l\in\mathbb{Z}$. Hence we have:
%$$
%\begin{array}{c}
%  (k^2+l^2-8\alpha\bar{\alpha})
%f_{k,l}=-8\overline{f_{-k,-l}}\alpha\bar{\beta}; \\
% (k^2+l^2-8\alpha\bar{\alpha})
%f_{-k,-l}=-8\overline{f_{k,l}}\alpha\bar{\beta}.  \\
%\end{array}
%$$
%Simple algebraic manipulation show us that, if $(k,l)\neq (0,0)$,
%then $k^2+l^2-16\alpha\bar{\alpha}=0$; on the other hand, from
%(\ref{popo}), we see that
%$\alpha\bar{\alpha}=\frac{n^2}{16}+\frac{m^2}{16}$ for some
%integers $n,m$. Then $k^2+l^2=n^2+m^2$ and $|f_{kl}|=|f_{-k-l}|$.
%In the case $k=l=0$, we have
%$f_{00}\bar{\alpha}=-\overline{f_{00}}\bar{\beta}$. Then $v$ is a
%Jacobi field if and only if $V$ is of the form
%
%with $V_0\in and
%$V_{k,l}$ is of the form
%$$V_{k,l}=\left(%
%\begin{array}{cc}
%  0 & f_{k,l} \\
%  -\overline{f_{-k,-l}} & 0\\
%\end{array}%
%\right)$$ with $|f_{kl}|=|f_{-k-l}|$.
%

\end{document}